\documentclass[11pt,reqno]{amsart}

\usepackage{amssymb}

\usepackage{amsfonts}

\usepackage{amsmath}

\usepackage{color}

\usepackage[all]{xy}

\usepackage{stackengine}

\newtheorem{theorem}{Theorem}[section]

\newtheorem{corollary}[theorem]{Corollary}

\newtheorem{lemma}[theorem]{Lemma}

\newtheorem{proposition}[theorem]{Proposition}

\newtheorem{Definition}[theorem]{Definition}

\newtheorem{Example}[theorem]{Example}

\newtheorem{Remark}[theorem]{Remark}

\newenvironment{remark}{\begin{Remark}\begin{em}}{\end{em}\end{Remark}}

\newenvironment{definition}{\begin{Definition}\begin{em}}{\end{em}\end{Definition}}

\DeclareMathOperator{\tr}{tr}

\address{Sejong Kim, Department of Mathematics, Chungbuk National University, Cheongju 28644, Korea}
\email{skim@chungbuk.ac.kr}

\address{Vatsalkumar N. Mer \\ Institute for Industrial and Applied Mathematics, Chungbuk National University, Cheongju 28644, Korea}
\email{vnm232657@gmail.com}

\begin{document}

\author{ Sejong Kim and Vatsalkumar N. Mer}

\title[New multivariable mean from nonlinear matrix equation]{New multivariable mean from nonlinear matrix equation associated to the harmonic mean}

\date{\today}
\maketitle

\begin{abstract}
Various multivariable means have been defined for positive definite matrices, such as the Cartan mean, Wasserstein mean, and R\'{e}nyi power mean. These multivariable means have corresponding matrix equations. In this paper, we consider the following non-linear matrix equation:
\begin{displaymath}
X = \left[ \sum_{i=1}^{n} w_{i} [ (1-t) X + t A_{i} ]^{-1} \right]^{-1},
\end{displaymath}
where $t \in (0,1]$. We prove that this equation has a unique solution and define a new mean, which we denote as $G_{t}(\omega; \mathbb{A})$.
We explore important properties of the mean $G_{t}(\omega; \mathbb{A})$ including the relationship with matrix power mean, and show that the mean $G_{t}(\omega; \mathbb{A})$ is monotone in the parameter $t$.
Finally, we connect the mean $G_{t}(\omega; \mathbb{A})$ to a barycenter for the log-determinant divergence.

\vspace{5mm}

\noindent {\bf Mathematics Subject Classification} (2020): 15A24, 15B48

\noindent {\bf Keywords}: Thompson metric, Ando-Li-Mathias's property, monotonicity, quantum divergence
\end{abstract}

\section{Introduction}

Several various types of multivariable means are found on the open convex cone $\mathbb{P}_{m}$ consisting of positive definite matrices such as the Cartan mean, Wasserstein mean, and R\'{e}nyi power mean. Correspondingly, each of these means is associated with specific matrix equations. For instance, the Cartan mean $\Lambda$ is defined as a unique minimizer of the weighted sum of squares of Riemannian trace distances to each point: for any $\mathbb{A} = (A_{1}, \ldots, A_{n}) \in \mathbb{P}_{m}^{n}$ and a positive probability vector $\omega = (w_{1}, \dots, w_{n})$
\begin{displaymath}
\Lambda(\omega; \mathbb{A}) = \underset{X \in \mathbb{P}_{m}}{\arg \min} \, \sum_{i=1}^{n} w_{i} d_{R}^{2}(X, A_{i}),
\end{displaymath}
where $d_{R}(A, B) = \Vert \log (A^{-1/2} B A^{-1/2}) \Vert_{2}$ denotes the Riemannian trace distance between $A$ and $B$.
We obtain from \cite{Ka} that the Cartan mean coincides with the unique positive definite solution $X$ of the following matrix equation, called the Karcher equation,
\begin{displaymath}
\sum_{i=1}^{n} w_{i} \log (X^{-1/2} A_{i} X^{-1/2}) = 0.
\end{displaymath}
An interesting approach of the Cartan mean is the power mean $P_{t} (\omega; \mathbb{A})$, defined as the unique solution $X \in \mathbb{P}_{m}$ of the nonlinear matrix equation
\begin{displaymath}
\displaystyle X = \sum_{i=1}^{n} w_{i} X \#_{t} A_{i}
\end{displaymath}
for $t \in (0,1]$. We define $P_{t} (\omega; \mathbb{A}) = P_{-t}(\omega; \mathbb{A}^{-1})^{-1}$ for $t \in [0,1)$. It has been shown in \cite{LL14, LP} that the power mean $P_{t} (\omega; \mathbb{A})$ converges to the Cartan mean $\Lambda(\omega; \mathbb{A})$ as $t \to 0$, and furthermore
\begin{equation} \label{E:P-mono}
P_{-1} = \mathcal{H} \leq \cdots \leq P_{-t} \leq P_{-s} \leq \cdots \leq \Lambda = P_{0} \leq \cdots \leq P_{s} \leq P_{t} \leq \cdots \leq P_{1} = \mathcal{A}
\end{equation}
for $0 < s \leq t \leq 1$, where $\mathcal{A}$ and $\mathcal{H}$ denote the weighted arithmetic and harmonic means, respectively.

The Wasserstein mean $\Omega$ is defined as the least squares mean for the Bures-Wasserstein metric: for any $\mathbb{A} = (A_{1}, \ldots, A_{n}) \in \mathbb{P}_{m}^{n}$ and $\omega = (w_{1}, \dots, w_{n}) \in \Delta_{n}$
\begin{displaymath}
\Omega(\omega; \mathbb{A}) = \underset{X \in \mathbb{P}_{m}}{\arg \min} \, \sum_{i=1}^{n} w_{i} d_{W}^{2}(X, A_{i}),
\end{displaymath}
where $d_{W}(A, B) = \left[ \tr (A + B - 2(A^{1/2} B A^{1/2})^{1/2}) \right]^{1/2}$ denotes the Bures-Wasserstein distance between $A$ and $B$.
Note from \cite{ABCM, BJL-1} that the Wasserstein mean coincides with the unique positive definite solution $X$ of the following matrix equation
\begin{equation} \label{E:Wasserstein}
X = \sum_{i=1}^{n} w_{i} (X^{1/2} A_{i} X^{1/2})^{1/2}.
\end{equation}
As a generalization of \eqref{E:Wasserstein}, the R\'{e}nyi power mean $\mathcal{R}_{t,z}$ for $0 < t \leq z < 1$ is defined in \cite{DF} as a unique positive definite solution of the matrix equation
\begin{displaymath}
X = \sum_{j=1}^{n} w_{j} \left( A_{j}^{\frac{1-t}{2z}} X^{\frac{t}{z}} A_{j}^{\frac{1-t}{2z}} \right)^{z}.
\end{displaymath}
where $Q_{t, z}(A, B) := \left( A^{\frac{1-t}{2z}} B^{\frac{t}{z}} A^{\frac{1-t}{2z}} \right)^{z}$ is the matrix version of the $t$-$z$ R\'{e}nyi relative entropy \cite{AD, MO}. One can see that $\mathcal{R}_{t,z} = \Omega$ when $t = z = 1/2$. See \cite{BJL-2, JK} for more properties about Wasserstein mean and R\'{e}nyi power mean.

As we have seen previously, multivariable means correspond to their own matrix equations. In this paper we consider the following nonlinear matrix equation
\begin{displaymath}
X = \left[ \sum_{i=1}^{n} w_{i} [ (1-t) X + t A_{i} ]^{-1} \right]^{-1},
\end{displaymath}
where $t \in (0,1]$. We obtain new multivariable mean by proving that the right-hand side of the above equation is a strict contraction for the Thompson metric. In Section 3, we explore the fundamental properties of the mean from in the view point of Ando-Li-Mathias's multivariable geometric mean and and we establish bounds in terms of the operator norm. One of the interesting results, appeared in Section 4, is the monotonicity for parameter $t$, which means that new multivariable mean interpolates the weighted arithmetic and harmonic means. Furthermore, we prove the relationship between $G_{t}$ and the matrix power mean $P_{t}$. Finally, we see that new multivariable mean is connected to the right mean for quantum divergence when $\omega$ is a uniformly-distributed probability vector. From the fact that new multivariable mean for $n = 2$ coincides with the metric geometric mean when $t = 1/2$ and $\omega = (1/2, 1/2)$, we give some open questions.

\section{Thompson metric}


The Thompson metric on $\mathbb{P}_{m}$ is defined by
$d(A, B) = \Vert \log(A^{-1/2} B A^{-1/2}) \Vert$. It is known that $d$ is a complete metric on $\mathbb{P}_{m}$ and that
\begin{equation} \label{E:Thompson}
d(A,B) = \max \{ \log M(B/A), \log M(A/B) \},
\end{equation}
where $M(B/A) = \inf \{\alpha > 0: B \leq \alpha A\} = \lambda_{1} (A^{-1/2} B A^{-1/2})$: see \cite{CPR, Nu1, Th}.
The geometric mean curve
\begin{equation} \label{E:geomean}
\displaystyle [0,1] \ni t \mapsto A \#_{t} B := A^{1/2} (A^{-1/2} B A^{-1/2})^{t} A^{1/2}.
\end{equation}
is a minimal geodesic from $A$ to $B$ for the Thompson metric. We call $A \#_{t} B$ the \emph{weighted geometric mean} of $A$ and $B$ on $\mathbb{P}_{m}$, and simply write $A \# B := A \#_{1/2} B$.

\begin{lemma} \cite{CPR,LL} \label{L:Thompson}
Basic properties of the Thompson metric on $\mathbb{P}_{m}$ include
\begin{itemize}
\item[(1)] $d(A, B) = d(A^{-1}, B^{-1}) = d(S A S^{*}, S B S^{*})$
for any $S \in \mathrm{GL};$
\item[(2)] $d(A \#_{t} B, C \#_{t} D) \leq (1-t) d(A,C) + t d(B,D)$ for any $t \in [0,1];$
\item[(3)] $d(A \#_{s} B, A \#_{t} B) = |s-t| d(A, B)$ for any $s, t \in [0,1].$
\end{itemize}
\end{lemma}

By Lemma \ref{L:Thompson} the geometric mean satisfies the following convexity for $d$:
\begin{equation} \label{E:continuity}
d(A \#_{s} B, C \#_{t} D) \leq (1-t) \, d(A, C) + t \, d(B, D) + |s-t| \, d(C, D)
\end{equation}
for $A, B, C, D \in \mathbb{P}_{m}$ and $s, t \in [0,1]$.

\begin{lemma} \label{L:nonexpansive}
The Thompson metric satisfies $d(A + B, C + D) \leq \max \{ d(A, C), d(B, D) \}$ for any $A, B, C, D \in \mathbb{P}_{m}$. So inductively we have for any $A_{i}, B_{i} \in \mathbb{P}_{m}, \ 1 \leq i \leq n$
\begin{displaymath}
d \left( \sum_{i=1}^{n} A_{i}, \sum_{i=1}^{n} B_{i} \right) \leq \underset{1 \leq i \leq n}{\max} \{ d(A_{i}, B_{i}) \}.
\end{displaymath}
\end{lemma}

The following is the additive contraction principle on $\mathbb{P}_{m}$ in \cite{LeeL}:
\begin{lemma} \label{L:contraction}
Let $A \in \mathbb{P}_{m}$. Then
\begin{displaymath}
d(A + X, A + Y) \leq \frac{\alpha}{\alpha + \beta} d(X, Y)
\end{displaymath}
for any $X, Y \in \mathbb{P}_{m}$, where $\alpha = \max \{ \lambda_{1}(X), \lambda_{1}(Y) \}$ and $\beta = \lambda_{m}(A)$.
\end{lemma}

\section{Multivariable means from nonlinear matrix equation}

Let $\mathbb{A} = (A_{1}, \ldots, A_{n}) \in \mathbb{P}_{m}^{n}$ and let $\omega = (w_{1}, \dots, w_{n}) \in \Delta_{n}$, the set of all positive probability vectors in $\mathbb{R}^{n}$.
We solve in this section the following nonlinear matrix equation
\begin{equation} \label{E:fixed}
X = \left[ \sum_{i=1}^{n} w_{i} [ (1-t) X + t A_{i} ]^{-1} \right]^{-1}
\end{equation}

\begin{proposition} \label{P:fixed}
Let $\mathbb{A} = (A_{1}, \ldots, A_{n}) \in \mathbb{P}_{m}^{n}$ and let $\omega = (w_{1}, \dots, w_{n}) \in \Delta_{n}$. Then the equation \eqref{E:fixed} has a unique positive definite solution for $t \in (0,1]$.
\end{proposition}

\begin{proof}
When $t = 1$, we obtain the weighted harmonic mean $\displaystyle X = \left[ \sum_{i=1}^{n} w_{i} A_{i}^{-1} \right]^{-1}$. Assume that $t \in (0,1)$. Set $\displaystyle f(X) = \left[ \sum_{i=1}^{n} w_{i} [ (1-t) X + t A_{i} ]^{-1} \right]^{-1}$ for $X \in \mathbb{P}_{m}$. By Lemma \ref{L:Thompson}, Lemma \ref{L:nonexpansive}, and Lemma \ref{L:contraction}
\begin{displaymath}
\begin{split}
\displaystyle d(f(X), f(Y)) & = d \left( \sum_{i=1}^{n} w_{i} [ (1-t) X + t A_{i} ]^{-1}, \sum_{i=1}^{n} w_{i} [ (1-t) Y + t A_{i} ]^{-1} \right) \\
& \leq \underset{1 \leq i \leq n}{\max} \, d((1-t) X + t A_{i}, (1-t) Y + t A_{i}) \\
& \leq \underset{1 \leq i \leq n}{\max} \, \frac{(1-t) \alpha}{(1-t) \alpha + t \beta_{i}} d(X, Y),
\end{split}
\end{displaymath}
where $\displaystyle \alpha = \max \{ \lambda_{1}(X), \lambda_{1}(Y) \}$, and $\beta_{i} = \lambda_{m}(A_{i})$. Note that
$$ \displaystyle 0 < \frac{(1-t) \alpha}{(1-t) \alpha + t \beta_{i}} < 1, $$
since $\beta_{i} > 0$ for all $i = 1, \dots, n$.
By the Banach fixed point theorem, there exists a unique solution $X \in \mathbb{P}_{m}$ of the equation $X = f(X)$.
\end{proof}

\begin{definition}
Let $\mathbb{A} = (A_{1}, \ldots, A_{n}) \in \mathbb{P}_{m}^{n}$ and let $\omega = (w_{1}, \dots, w_{n}) \in \Delta_{n}$. We define the matrix mean $G_{t}(\omega; \mathbb{A})$ for $t \in (0,1]$ as the unique positive definite solution $X$ of the equation \eqref{E:fixed}.
\end{definition}

\begin{remark} \label{R:Banach}
The function $\displaystyle f(X) = \left[ \sum_{i=1}^{n} w_{i} [ (1-t) X + t A_{i} ]^{-1} \right]^{-1}$ for $X \in \mathbb{P}_{m}$ is monotone increasing. Moreover, the Banach fixed point theorem yields that
\begin{displaymath}
\lim_{k \to \infty} f^{k}(Z) = G_{t}(\omega; \mathbb{A}) \ \textrm{for any} \ Z \in \mathbb{P}_{m}.
\end{displaymath}
\end{remark}

\begin{remark} \label{R:alternative}
The unique positive definite solution $X$ of the equation \eqref{E:fixed} satisfies the following alternative nonlinear equation
\begin{displaymath}
\displaystyle I = \sum_{i=1}^{n} w_{i} X^{1/2} \left[ (1-t) X + t A_{i} \right]^{-1} X^{1/2} = \sum_{i=1}^{n} w_{i} \left[ (1-t) I + t X^{-1/2} A_{i} X^{-1/2} \right]^{-1}.
\end{displaymath}
This can be written as
\begin{equation} \label{E:Req}
\displaystyle \mathcal{R}_{\frac{1-t}{t}} \left( \omega; X^{-1/2} \mathbb{A} X^{-1/2} \right) = I,
\end{equation}
where $\displaystyle X^{-1/2} \mathbb{A} X^{-1/2} := (X^{-1/2} A_{1} X^{-1/2}, \ldots, X^{-1/2} A_{n} X^{-1/2})$. Here, $\mathcal{R}_{\frac{1-t}{t}}$ denotes the resolvent mean of parameter $\displaystyle \frac{1-t}{t}$ \cite{KLL}. In other words, $G_{t}(\omega; \mathbb{A})$ is the unique positive definite solution $X$ of the equation \eqref{E:Req}.
\end{remark}

We see fundamental properties of $G_{t}$ in perspective on the Ando-Li-Mathias's properties of multivariable geometric mean.
Let $p \in \mathbb{R}$, $\sigma$ a permutation on $\{ 1, \dots, n \}$, and $S \in \mathrm{GL}_{m}$, the general linear group of all $m \times m$ invertible matrices. For convenience, we denote
\begin{displaymath}
\begin{split}
\mathbb{A}^{p} & = (A_{1}^{p}, \ldots, A_{n}^{p}) \\
S \mathbb{A} S^{*} & = (S A_{1} S^{*}, \ldots, S A_{n} S^{*}) \\
\mathbb{A}_{\sigma} & = (A_{\sigma(1)}, \ldots, A_{\sigma(n)}) \\
\omega_{\sigma} & = (w_{\sigma(1)}, \ldots, w_{\sigma(n)}).
\end{split}
\end{displaymath}

\begin{theorem} \label{T:properties}
Let $\mathbb{A} = (A_{1}, \ldots, A_{n}), \mathbb{B} = (B_{1}, \ldots, B_{n}) \in \mathbb{P}_{m}^{n}$, and let $\omega = (w_{1}, \dots, w_{n}) \in \Delta_{n}$. Then for $t \in (0,1]$
\begin{itemize}
\item[(1)] $\mathrm{(Idempotency)}$ $G_{t}(\omega; A, \dots, A) = A$ for any $A \in \mathbb{P}_{m}$;

\item[(2)] $\mathrm{(Homogeneity)}$ $G_{t}(\omega; \alpha \mathbb{A}) = \alpha G_{t}(\omega; \mathbb{A})$ for any $\alpha > 0$;

\item[(3)] $\mathrm{(Permutation \ invariance)}$ $G_{t}(\omega_{\sigma}; \mathbb{A}_{\sigma}) = G_{t}(\omega; \mathbb{A})$ for any permutation $\sigma$ on $\{ 1, \dots, n \}$;

\item[(4)] $\mathrm{(Monotonicity)}$ $G_{t}(\omega; \mathbb{A}) \leq G_{t}(\omega; \mathbb{B})$ whenever $A_{i} \leq B_{i}$ for all $i = 1, \dots, n$;

\item[(5)] $\mathrm{(Continuity)}$ $G_{t}(\omega; \cdot): \mathbb{P}_{m}^{n} \to \mathbb{P}_{m}$ is non-expansive for the Thompson metric:
\begin{displaymath}
d(G_{t}(\omega; \mathbb{A}), G_{t}(\omega; \mathbb{B})) \leq \underset{1 \leq i \leq n}{\max} \, d(A_{i}, B_{i});
\end{displaymath}

\item[(6)] $\mathrm{(Congruence \ invariance)}$ $G_{t}(\omega; S \mathbb{A} S^{*}) = S G_{t}(\omega; \mathbb{A}) S^{*}$ for any $S \in \mathrm{GL}_{m}$;


\item[(7)] $\mathrm{(Self-duality)}$ $G_{1-t}(\omega; \mathbb{A}^{-1}) = G_{t}(\omega; \mathbb{A})^{-1}$ for $t \in (0,1)$;

\item[(8)] $\mathrm{(Arithmetic-G-harmonic \ mean \ inequalities)}$
\begin{displaymath}
\mathcal{H}(\omega; \mathbb{A}) = \left[ \sum_{i=1}^{n} w_{i} A_{i}^{-1} \right]^{-1} \leq G_{t}(\omega; \mathbb{A}) \leq \sum_{i=1}^{n} w_{i} A_{i} = \mathcal{A}(\omega; \mathbb{A}).
\end{displaymath}
\end{itemize}
\end{theorem}

\begin{proof}
Note that $\displaystyle G_{t=1}(\omega; \mathbb{A}) = \left[ \sum_{i=1}^{n} w_{i} A_{i}^{-1} \right]^{-1}$ holds all properties, so we assume $t \in (0,1)$ in the following. Items (1)-(3) are obvious.
\begin{itemize}



\item[(4)] Let $X = G_{t}(\omega; \mathbb{A})$, and let $\displaystyle g(Z) = \left[ \sum_{i=1}^{n} w_{i} [ (1-t) Z + t B_{i} ]^{-1} \right]^{-1}$ for $Z \in \mathbb{P}_{m}$. Assume that $A_{i} \leq B_{i}$ for all $i = 1, \dots, n$. Then
\begin{displaymath}
X = \left[ \sum_{i=1}^{n} w_{i} [ (1-t) X + t A_{i} ]^{-1} \right]^{-1} \leq \left[ \sum_{i=1}^{n} w_{i} [ (1-t) X + t B_{i} ]^{-1} \right]^{-1} = g(X).
\end{displaymath}
Inductively, $X \leq g(X) \leq \cdots \leq g^{k}(X)$ for all $k \in \mathbb{N}$,
and hence, by Remark \ref{R:Banach} $X \leq G_{t}(\omega; \mathbb{B})$.

\item[(5)] It is known from \cite{KL} that every monotonic and subhomogeneous map is non-expansive for the Thompson metric.
From item (2) the homogeneity of $G$ implies the sub-homogeneity, and hence, $G$ is non-expansive by item (4).

\item[(6)] Let $X = G_{t}(\omega; S \mathbb{A} S^{*})$ for any $S \in \mathrm{GL}_{m}$.
By Proposition \ref{P:fixed} we have
\begin{displaymath}
\begin{split}
X & = \left[ \sum_{i=1}^{n} w_{i} [ (1-t) X + t S A_{i} S^{*} ]^{-1} \right]^{-1} \\
& = \left[ (S^{*})^{-1} \left[ \sum_{i=1}^{n} w_{i} [ (1-t) S^{-1} X (S^{*})^{-1} + t A_{i} ]^{-1} \right] S^{-1} \right]^{-1} \\
& = S \left[ \sum_{i=1}^{n} w_{i} [ (1-t) S^{-1} X (S^{*})^{-1} + t A_{i} ]^{-1} \right]^{-1} S^{*}.
\end{split}
\end{displaymath}
Thus, $S^{-1} X (S^{*})^{-1} = G_{t}(\omega; \mathbb{A})$, so it is proved.

\item[(7)] Let $X = G_{t}(\omega; \mathbb{A})$. From Remark \ref{R:alternative}
\begin{displaymath}
\displaystyle I = \sum_{i=1}^{n} w_{i} \left[ (1-t) I + t X^{-1/2} A_{i} X^{-1/2} \right]^{-1}.
\end{displaymath}
Applying the formula $(A + B)^{-1} = A^{-1} - A^{-1} (A^{-1} + B^{-1})^{-1} A^{-1}$ in \cite{Zh} to the above equation yields
\begin{displaymath}
\begin{split}
\displaystyle I & = \sum_{i=1}^{n} w_{i} \left\{ \frac{1}{1-t} I - \frac{1}{(1-t)^{2}} \left[ \frac{1}{1-t} I + \frac{1}{t} X^{1/2} A_{i}^{-1} X^{1/2} \right]^{-1} \right\} \\
& = \frac{1}{1-t} I - \frac{1}{(1-t)^{2}} \sum_{i=1}^{n} w_{i} \left[ \frac{1}{1-t} I + \frac{1}{t} X^{1/2} A_{i}^{-1} X^{1/2} \right]^{-1} \\
& = \frac{1}{1-t} I - \frac{t}{1-t} \sum_{i=1}^{n} w_{i} \left[ t I + (1-t) X^{1/2} A_{i}^{-1} X^{1/2} \right]^{-1}.
\end{split}
\end{displaymath}
Equivalently, we have
\begin{displaymath}
\displaystyle I = \sum_{i=1}^{n} w_{i} \left[ t I + (1-t) X^{1/2} A_{i}^{-1} X^{1/2} \right]^{-1}.
\end{displaymath}
Thus, $X^{-1} = G_{1-t}(\omega; \mathbb{A}^{-1})$.

\item[(8)] Let $X = G_{t}(\omega; \mathbb{A})$. Then by the arithmetic-harmonic mean inequality
\begin{displaymath}
X = \left[ \sum_{i=1}^{n} w_{i} [ (1-t) X + t A_{i} ]^{-1} \right]^{-1} \leq \sum_{i=1}^{n} w_{i} [ (1-t) X + t A_{i} ] = (1-t) X + t \sum_{i=1}^{n} w_{i} A_{i}.
\end{displaymath}
So we obtain that $\displaystyle X \leq \sum_{i=1}^{n} w_{i} A_{i}$.

Replacing $t$ and $A_{i}$ by $1-t$ and $A_{i}^{-1}$ for all $i = 1, \dots, n$ in the previous result, we have
\begin{displaymath}
G_{1-t}(\omega; \mathbb{A}^{-1}) \leq \sum_{i=1}^{n} w_{i} A_{i}^{-1}.
\end{displaymath}
By item (7) and taking inverse on both sides, we obtain $\displaystyle G_{t}(\omega; \mathbb{A}) \geq \left[ \sum_{i=1}^{n} w_{i} A_{i}^{-1} \right]^{-1}$.
\end{itemize}
\end{proof}

\begin{remark} \label{R:log-Euclidean}
From Theorem \ref{T:properties} (8), the multivariable mean $G_{t}$ is the Lie-Trotter mean \cite{HK}. As a consequence,
\begin{equation} \label{E:limit}
\lim_{p \to 0} G_{t}(\omega; A_{1}^{p}, \dots, A_{n}^{p})^{1/p} = \exp \left( \sum_{i=1}^{n} w_{i} \log A_{i} \right),
\end{equation}
where the right-hand side is known as the log-Euclidean mean.
\end{remark}

\begin{proposition} \label{P:more-prop}
Let $\mathbb{A} = (A_{1}, \ldots, A_{n}) \in \mathbb{P}_{m}^{n}$, and $\omega = (w_{1}, \dots, w_{n}) \in \Delta_{n}$. Then for $t \in (0,1]$
\begin{itemize}
\item[(1)] $\displaystyle \Vert G_{1-t}(\omega; \mathbb{A}) \Vert \leq \left[ \sum_{i=1}^{n} w_{i} \Vert A_{i} \Vert^{t} \right]^{1/t}$ for an operator norm $\Vert \cdot \Vert$,
\item[(2)] $\Phi(G_{t}(\omega; \mathbb{A})) \leq G_{t}(\omega; \Phi(\mathbb{A}))$ for any positive linear map $\Phi$, where $$ \Phi(\mathbb{A}) := (\Phi(A_{1}), \dots, \Phi(A_{n})). $$
\end{itemize}
\end{proposition}

\begin{proof}
Let $X = G_{t}(\omega; \mathbb{A})$.
\begin{itemize}
\item[(1)] By \cite[Theorem 3]{BG} and the sub-multiplicativity of operator norm,
\begin{equation} \label{E:norm}
\Vert A \#_{t} B \Vert \leq \Vert A^{1-t} B^{t} \Vert \leq \Vert A \Vert^{1-t} \Vert B^{t} \Vert
\end{equation}
for $A, B \in \mathbb{P}_{m}$. Since
\begin{displaymath}
X^{-1} = \sum_{i=1}^{n} w_{i} ((1-t) X + t A_{i})^{-1} \leq \sum_{i=1}^{n} w_{i} (X \#_{t} A_{i})^{-1} = \sum_{i=1}^{n} w_{i} (X^{-1} \#_{t} A_{i}^{-1})
\end{displaymath}
by the arithmetic-geometric mean inequality, we have from \eqref{E:norm}
\begin{displaymath}
\Vert X^{-1} \Vert \leq \sum_{i=1}^{n} w_{i} \Vert X^{-1} \#_{t} A_{i}^{-1} \Vert \leq \sum_{i=1}^{n} w_{i} \Vert X^{-1} \Vert^{1-t} \Vert A_{i}^{-1} \Vert^{t}.
\end{displaymath}
Solving for $\Vert X^{-1} \Vert$, we get
\begin{displaymath}
\Vert X^{-1} \Vert = \Vert G_{t}(\omega; \mathbb{A})^{-1} \Vert \leq \left[ \sum_{i=1}^{n} w_{i} \Vert A_{i}^{-1} \Vert^{t} \right]^{1/t}.
\end{displaymath}
Applying Theorem \ref{T:properties} (7) and replacing $A_{i}$ by $A_{i}^{-1}$ we obtain the desired inequality.

\item[(2)] Set $Y = \Phi(X)$ for any positive linear map $\Phi$.
Then we have from \cite[Theorem 4.1.5]{Bh}
\begin{displaymath}
\begin{split}
Y & = \Phi(X) = \Phi \left( \left[ \sum_{i=1}^{n} w_{i} [ (1-t) X + t A_{i} ]^{-1} \right]^{-1} \right) \\
& \leq \left[ \sum_{i=1}^{n} w_{i} [ \Phi ((1-t) X + t A_{i}) ]^{-1} \right]^{-1} = \left[ \sum_{i=1}^{n} w_{i} [ (1-t) \Phi(X) + t \Phi(A_{i}) ]^{-1} \right]^{-1} =: h(Y).
\end{split}
\end{displaymath}
Since the map $h$ is monotone increasing, $Y \leq h(Y) \leq h^{2}(Y) \leq \cdots \leq h^{k}(Y)$ for all $k \geq 2$. Taking limit as $k \to \infty$ and applying Remark \ref{R:Banach} yields
\begin{displaymath}
\Phi(G_{t}(\omega; \mathbb{A})) = Y \leq \lim_{k \to \infty} h^{k}(Y) = G_{t}(\omega; \Phi(\mathbb{A})).
\end{displaymath}
\end{itemize}
\end{proof}

\section{Inequalities and monotonicity on parameters}

In this section we focus on parameters of the multivariable mean $G_{t}$, and see that it interpolates backward the arithmetic mean and harmonic mean.

\begin{theorem} \label{T:monotone}
For $0 < s \leq t \leq 1$
\begin{displaymath}
G_{s}(\omega; \mathbb{A}) \geq G_{t}(\omega; \mathbb{A}).
\end{displaymath}
\end{theorem}

\begin{proof}
Let $X = G_{s}(\omega; \mathbb{A})$ and $Y = G_{t}(\omega; \mathbb{A})$. Assume that $0 < s \leq t \leq 1$. Set $\displaystyle f(Z) = \left[ \sum_{i=1}^{n} w_{i} [ (1-s) Z + s A_{i} ]^{-1} \right]^{-1}$ for $Z \in \mathbb{P}_{m}$. Then
\begin{displaymath}
\begin{split}
f(Z) & = \left[ \sum_{i=1}^{n} w_{i} \left\{ \left( 1 - \frac{s}{t} \right) Z + \frac{s}{t} ((1-t) Z + t A_{i}) \right\}^{-1} \right]^{-1} \\
& \geq \left[ \sum_{i=1}^{n} w_{i} \left\{ \left( 1 - \frac{s}{t} \right) Z^{-1} + \frac{s}{t} ((1-t) Z + t A_{i})^{-1} \right\} \right]^{-1} \\
& = \left[ \left( 1 - \frac{s}{t} \right) Z^{-1} + \frac{s}{t} \sum_{i=1}^{n} w_{i} ((1-t) Z + t A_{i})^{-1} \right]^{-1}.
\end{split}
\end{displaymath}
The inequality follows from the two-variable arithmetic-harmonic mean inequality such that $[ (1 - \lambda) A + \lambda B ]^{-1} \leq (1 - \lambda) A^{-1} + \lambda B^{-1}$ for any $A, B \in \mathbb{P}_{m}$ and $\lambda \in [0,1]$.
Substituting $Z = Y$ in the above and applying Proposition \ref{P:fixed} to $Y$ yields
\begin{displaymath}
f(Y) \geq \left[ \left( 1 - \frac{s}{t} \right) Y^{-1} + \frac{s}{t} \sum_{i=1}^{n} w_{i} ((1-t) Y + t A_{i})^{-1} \right]^{-1} = \left[ \left( 1 - \frac{s}{t} \right) Y^{-1} + \frac{s}{t} Y^{-1} \right]^{-1} = Y.
\end{displaymath}
Since the map $f$ is monotone increasing in Remark \ref{R:Banach}, we have inductively $Y \leq f(Y) \leq \cdots \leq f^{k}(Y)$ for all $k \in \mathbb{N}$. Taking the limit as $k \to \infty$ implies that
\begin{displaymath}
Y \leq \lim_{k \to \infty} f^{k}(Y) = X.
\end{displaymath}
\end{proof}

\begin{theorem} \label{T:limit}
Let $\mathbb{A} = (A_{1}, \ldots, A_{n}) \in \mathbb{P}_{m}^{n}$, and let $\omega = (w_{1}, \dots, w_{n}) \in \Delta_{n}$. Then
\begin{displaymath}
\lim_{t \to 0^{+}} G_{t}(\omega; \mathbb{A}) = \sum_{i=1}^{n} w_{i} A_{i}.
\end{displaymath}
\end{theorem}

\begin{proof}
Let $X_{t} = G_{t}(\omega; \mathbb{A})$ for $t \in (0,1)$. By Theorem \ref{T:monotone} and Theorem \ref{T:properties} (8), $X_{t}$ is monotonic and bounded. So it converges as $t \to 0^{+}$, say $\displaystyle X_{0} = \lim_{t \to 0^{+}} G_{t}(\omega; \mathbb{A})$. By Remark \ref{R:alternative}
\begin{displaymath}
\displaystyle I = \sum_{i=1}^{n} w_{i} \left[ (1-t) I + t X_{t}^{-1/2} A_{i} X_{t}^{-1/2} \right]^{-1}.
\end{displaymath}
It can be rewritten equivalently as
\begin{displaymath}
\displaystyle \sum_{i=1}^{n} w_{i} \frac{\left[ (1-t) I + t X_{t}^{-1/2} A_{i} X_{t}^{-1/2} \right]^{-1} - I}{t} = 0.
\end{displaymath}
Taking the limit as $t \to 0^{+}$ implies that
\begin{displaymath}
\displaystyle \sum_{i=1}^{n} w_{i} [I - X_{0}^{-1/2} A_{i} X_{0}^{-1/2}] = 0.
\end{displaymath}
Therefore, $\displaystyle X_{0} = \sum_{i=1}^{n} w_{i} A_{i}$.
\end{proof}

\begin{remark}
Theorem \ref{T:limit} allows us to define $G_{0^{+}}(\omega; \mathbb{A})$ as the weighted arithmetic mean, and furthermore, all properties in Theorem \ref{T:properties} and Proposition \ref{P:more-prop} hold for $t=0$.
We can also see from Theorem \ref{T:monotone} and Theorem \ref{T:limit} that $G_{t}$ interpolates backward the arithmetic mean and harmonic mean: for $0 < s \leq t \leq 1$
\begin{displaymath}
G_{0^{+}} = \mathcal{A} \geq G_{s} \geq G_{t} \geq \mathcal{H} = G_{1}.
\end{displaymath}
\end{remark}

The matrix power mean $P_{t}$ for $t \in [-1,1]$ interpolates the weighted arithmetic, Cartan, and harmonic means. So it is an interesting question to find the relationship between new multivariable mean $G_{t}$ and matrix power mean.

\begin{theorem} \label{T:G-Power-ineq}
Let $\mathbb{A} = (A_{1}, \ldots, A_{n}) \in \mathbb{P}_{m}^{n}$, and let $\omega = (w_{1}, \dots, w_{n}) \in \Delta_{n}$. Then for $t \in [0,1]$
\begin{equation} \label{E:G-Power-ineq}
G_{t}(\omega; \mathbb{A}) \geq P_{-t}(\omega; \mathbb{A}),
\end{equation}
where $P_{-t}(\omega; \mathbb{A}) := P_{t}(\omega; \mathbb{A}^{-1})^{-1}$.
\end{theorem}

\begin{proof}
Since $G_{t}(\omega; \mathbb{A}) \to \mathcal{A}(\omega; \mathbb{A})$ as $t \to 0^{+}$ by Theorem \ref{T:limit} and $P_{-t}(\omega; \mathbb{A}) \to \Lambda(\omega; \mathbb{A})$, the inequality \eqref{E:G-Power-ineq} holds when $t = 0$.

Let $X = G_{t}(\omega; \mathbb{A})^{-1}$ for $t \in (0,1]$. Then $X^{-1} = G_{t}(\omega; \mathbb{A})$, and by Proposition \ref{P:fixed}
\begin{displaymath}
X^{-1} = \left[ \sum_{i=1}^{n} w_{i} [ (1-t) X^{-1} + t A_{i} ]^{-1} \right]^{-1}.
\end{displaymath}
By the geometric-harmonic mean inequality
\begin{displaymath}
X = \sum_{i=1}^{n} w_{i} [ (1-t) X^{-1} + t A_{i} ]^{-1} \leq \sum_{i=1}^{n} w_{i} X \#_{t} A_{i}^{-1} =: g(X).
\end{displaymath}
Since the map $g: \mathbb{P}_{m} \to \mathbb{P}_{m}$ is monotone increasing, $X \leq g(X) \leq \cdots \leq g^{k}(X)$ for all $k \geq 1$. Taking limit as $k \to \infty$ yields
\begin{displaymath}
X \leq \lim_{k \to \infty} g^{k}(X) = P_{t}(\omega; \mathbb{A}^{-1}).
\end{displaymath}
Therefore, $X^{-1} = G_{t}(\omega; \mathbb{A}) \geq P_{t}(\omega; \mathbb{A}^{-1})^{-1} = P_{-t}(\omega; \mathbb{A})$.
\end{proof}


\begin{remark}
Taking limit as $t \to 0^{+}$ in \eqref{E:G-Power-ineq} and applying Theorem \ref{T:limit} we obtain the arithmetic-Cartan mean inequality:
\begin{displaymath}
\Lambda(\omega; \mathbb{A}) \leq \sum_{i=1}^{n} w_{i} A_{i}.
\end{displaymath}
\end{remark}

We provide another lower bound of $G_{t}$ and upper bound of $G_{1-t}$ for $t \in [0,1]$ by using Jensen-type inequalities in \cite{HP}: for a bounded linear operator $Z$ such that $Z^{-1}$ is a contraction
\begin{equation} \label{E:Jensen}
(Z^{*} A Z)^{t} \leq Z^{*} A^{t} Z.
\end{equation}

\begin{proposition} \label{P:bounds}
Let $\mathbb{A} = (A_{1}, \ldots, A_{n}) \in \mathbb{P}_{m}^{n}$, and let $\omega = (w_{1}, \dots, w_{n}) \in \Delta_{n}$. Then for $t \in [0,1]$
\begin{equation} \label{E:lower-bound}
G_{t}(\omega; \mathbb{A}) \geq \lambda_{\min}^{1-t} \left( \sum_{i=1}^{n} w_{i} A_{i}^{-t} \right)^{-1},
\end{equation}
and
\begin{equation} \label{E:upper-bound}
G_{1-t}(\omega; \mathbb{A}) \leq \lambda_{\max}^{1-t} \sum_{i=1}^{n} w_{i} A_{i}^{t},
\end{equation}
where $\lambda_{\min} := \min \{ \lambda_{m}(A_{i}): 1 \leq i \leq n \}$ and $\lambda_{\max} := \max \{ \lambda_{1}(A_{i}): 1 \leq i \leq n \}$.
\end{proposition}

\begin{proof}
Since $\lambda_{\min} I \leq A_{i} \leq \lambda_{\max} I$ for all $i$, we obtain $\displaystyle \lambda_{\min} I \leq \sum_{i=1}^{n} w_{i} A_{i} \leq \lambda_{\max} I$ so \eqref{E:lower-bound} and \eqref{E:upper-bound} hold when $t = 0$ and $t = 1$ by Theorem \ref{T:limit} and the arithmetic-harmonic mean inequality.

\begin{itemize}
  \item[(i)] Assume that $X = G_{t}(\omega; \mathbb{A}) \geq I$ for $t \in (0,1)$. Then by the arithmetic-geometric mean inequality and the self-duality of metric geometric mean
  \begin{displaymath}
  X^{-1} = \sum_{i=1}^{n} w_{i} ((1-t) X + t A_{i})^{-1} \leq \sum_{i=1}^{n} w_{i} (X \#_{t} A_{i})^{-1} = \sum_{i=1}^{n} w_{i} (X^{-1} \#_{t} A_{i}^{-1}).
  \end{displaymath}
  Taking congruence transformation by $X^{1/2}$ yields $\displaystyle I \leq \sum_{i=1}^{n} w_{i} (X^{1/2} A_{i}^{-1} X^{1/2})^{t}$. By Jensen-type inequality \eqref{E:Jensen}, we have
  \begin{displaymath}
  (X^{1/2} A_{i}^{-1} X^{1/2})^{t} \leq X^{1/2} A_{i}^{-t} X^{1/2}.
  \end{displaymath}
  Thus, $\displaystyle I \leq \sum_{i=1}^{n} w_{i} X^{1/2} A_{i}^{-t} X^{1/2}$, and hence, $\displaystyle X^{-1} \leq \sum_{i=1}^{n} w_{i} A_{i}^{-t}$. That is,
  \begin{equation} \label{E:lower}
  G_{t}(\omega; \mathbb{A}) \geq \left( \sum_{i=1}^{n} w_{i} A_{i}^{-t} \right)^{-1}.
  \end{equation}

  Since $\hat{A}_{i} := \lambda_{\min}^{-1} A_{i} \geq I$ for all $i$, we have from Theorem \ref{T:properties} (8)
  \begin{displaymath}
  G_{t}(\omega; \hat{A}_{1}, \dots, \hat{A}_{n}) \geq \left( \sum_{i=1}^{n} w_{i} \hat{A}_{i}^{-1} \right)^{-1} \geq I
  \end{displaymath}
  Applying \eqref{E:lower} with $\hat{A}_{i}$ for all $i$ and using Theorem \ref{T:properties} (2)
  \begin{displaymath}
  \lambda_{\min}^{-1} G_{t}(\omega; \mathbb{A}) = G_{t}(\omega; \hat{A}_{1}, \dots, \hat{A}_{n}) \geq \left( \sum_{i=1}^{n} w_{i} \hat{A}_{i}^{-t} \right)^{-1} = \lambda_{\min}^{-t} \left( \sum_{i=1}^{n} w_{i} A_{i}^{-t} \right)^{-1}.
  \end{displaymath}
  Simplifying it, we obtain \eqref{E:lower-bound}.

  \item[(ii)] Assume that $Y = G_{1-t}(\omega; \mathbb{A}) \leq I$ for $t \in (0,1)$. Then by Theorem \ref{T:properties} (7) and the geometric-harmonic mean inequality
  \begin{displaymath}
  Y = \sum_{i=1}^{n} w_{i} ((1-t) Y^{-1} + t A_{i}^{-1})^{-1} \leq \sum_{i=1}^{n} w_{i} (Y \#_{t} A_{i}).
  \end{displaymath}
  Taking congruence transformation by $Y^{-1/2}$ and applying \eqref{E:Jensen} yield \begin{displaymath}
  I \leq \sum_{i=1}^{n} w_{i} (Y^{-1/2} A_{i} Y^{-1/2})^{t} \leq \sum_{i=1}^{n} w_{i} Y^{-1/2} A_{i}^{t} Y^{-1/2}.
  \end{displaymath}
  Thus,
  \begin{equation} \label{E:upper}
  G_{1-t}(\omega; \mathbb{A}) = Y \leq \sum_{i=1}^{n} w_{i} A_{i}^{t}
  \end{equation}

  Since $\bar{A}_{i} := \lambda_{\max}^{-1} A_{i} \leq I$ for all $i$, we have from Theorem \ref{T:properties} (8)
  \begin{displaymath}
  G_{1-t}(\omega; \bar{A}_{1}, \dots, \bar{A}_{n}) \leq \sum_{i=1}^{n} w_{i} \bar{A}_{i} \leq I.
  \end{displaymath}
  Applying \eqref{E:upper} with $\bar{A}_{i}$ for all $i$ and using Theorem \ref{T:properties} (2) we have
  \begin{displaymath}
  \lambda_{\max}^{-1} G_{1-t}(\omega; \mathbb{A}) = G_{1-t}(\omega; \bar{A}_{1}, \dots, \bar{A}_{n}) \leq \sum_{i=1}^{n} w_{i} \bar{A}_{i}^{t} = \lambda_{\max}^{-t} \sum_{i=1}^{n} w_{i} A_{i}^{t}.
  \end{displaymath}
  Simplifying it, we obtain \eqref{E:upper-bound}.
\end{itemize}
\end{proof}

By the proofs (i) and (ii) of Proposition \ref{P:bounds} we obtain the following.
\begin{corollary}
Let $\mathbb{A} = (A_{1}, \ldots, A_{n}) \in \mathbb{P}_{m}^{n}$, and let $\omega = (w_{1}, \dots, w_{n}) \in \Delta_{n}$. For $t \in [0,1]$
\begin{itemize}
  \item[(i)] if $G_{t}(\omega; \mathbb{A}) \geq I$ then $\displaystyle G_{t}(\omega; \mathbb{A}) \geq \left( \sum_{i=1}^{n} w_{i} A_{i}^{-t} \right)^{-1}$, and
  \item[(ii)] if $G_{t}(\omega; \mathbb{A}) \leq I$ then $\displaystyle G_{t}(\omega; \mathbb{A}) \leq \sum_{i=1}^{n} w_{i} A_{i}^{1-t}$.
\end{itemize}
\end{corollary}

\begin{remark}
By Theorem \ref{T:properties} (7) and Theorem \ref{T:limit}
\begin{displaymath}
G_{1-t}(\omega; \mathbb{A}) = G_{t}(\omega; \mathbb{A}^{-1})^{-1}
\end{displaymath}
for $t \in [0,1]$.
Then one can see that $G_{1-t}(\omega; \mathbb{A})$ for $t \in (0,1]$ is the unique positive definite solution $X$ of
\begin{displaymath}
X = \sum_{i=1}^{n} w_{i} ((1-t) X^{-1} + t A_{i}^{-1})^{-1}.
\end{displaymath}
For $0 \leq s \leq t \leq 1$, Theorem \ref{T:monotone} yields
\begin{displaymath}
G_{1-t}(\omega; \mathbb{A}) = G_{t}(\omega; \mathbb{A}^{-1})^{-1} \geq G_{s}(\omega; \mathbb{A}^{-1})^{-1} = G_{1-s}(\omega; \mathbb{A}).
\end{displaymath}
Furthermore, one can see from \eqref{E:G-Power-ineq} that for $t \in [0,1]$
\begin{displaymath}
G_{1-t}(\omega; \mathbb{A}) = G_{t}(\omega; \mathbb{A}^{-1})^{-1} \leq P_{t}(\omega; \mathbb{A})
\end{displaymath}
\end{remark}


\section{Quantum divergence and barycenters}

For any $\alpha \in (-1,1)$ and $A, B \in \mathbb{P}_{m}$, the \emph{log-determinant $\alpha$-divergence} $D_{\alpha} (A | B)$ is defined by
\begin{equation} \label{E:log-det divergence}
D_{\alpha}(A | B) := \frac{4}{1 - \alpha^2} \log \frac{\det \left( \frac{1 - \alpha}{2} A + \frac{1 + \alpha}{2} B \right)}{(\det A)^{(1 - \alpha)/2} (\det B)^{(1 + \alpha)/2}}.
\end{equation}
One can see from \eqref{E:log-det divergence} that
\begin{displaymath}
\begin{split}
D_{-1} (A | B) & := \lim_{\alpha \to -1} D_{\alpha} (A | B) = \tr (A^{-1} B - I) - \log \det (A^{-1} B), \\
D_{1} (A | B) & := \lim_{\alpha \to 1} D_{\alpha} (A | B) = \tr (B^{-1} A - I) - \log \det (B^{-1} A).
\end{split}
\end{displaymath}
This is a one-parameter family of divergences which is related with the Stein's loss.
See \cite{CM} for more information.

\begin{definition} \label{D:q-divergence}
A smooth function $\Phi: \mathbb{P}_{m} \times \mathbb{P}_{m} \to [0, \infty)$ is a quantum divergence if it satisfies the following conditions:
\begin{enumerate}
\item $\Phi(A,B) \geq 0$, and $\Phi(A,B) = 0$ if and only if $A=B$.

\item The first derivative $D \Phi$ with respect to the second variable vanishes on the diagonal. That is,
\begin{displaymath}
\left. D \Phi(A, B) \right|_{B=A} = 0.
\end{displaymath}

\item The second derivative $D^{2} \Phi$ is non-negative on the diagonal. That is,
\begin{displaymath}
\left. \frac{\partial^2 \Phi}{\partial B^2} (A, B) \right|_{B=A} (X,X) \geq 0, \quad \forall X \in \mathbb{H}_m.
\end{displaymath}
\end{enumerate}
\end{definition}

\begin{theorem}
For any $\alpha \in (-1,1)$, the log-determinant $\alpha$-divergence $D_{\alpha} (A | B): \mathbb{P}_{m} \times \mathbb{P}_{m} \rightarrow [0,\infty)$ is  a quantum divergence.
\end{theorem}

\begin{proof}
The first condition of Definition \ref{D:q-divergence} for the log-determinant $\alpha$-divergence $D_{\alpha}$ has been proved in \cite[Proposition 3.5]{CM}.

\noindent
(2) For $X \in \mathbb{H}_m$ we have
\begin{displaymath}
\frac{\partial D_{\alpha} (A | B)}{\partial B} (X) = \frac{2}{1 - \alpha} \tr \left[ (- B^{-1} + (\frac{1 - \alpha}{2} A + \frac{1 + \alpha}{2} B)^{-1}) X \right].
\end{displaymath}
When $B=A$, we get $\displaystyle \left. \frac{\partial D_{\alpha} (A | B)}{\partial B} \right|_{B=A} (X) = 0$.

\noindent
(3) For $X \in \mathbb{H}_m$ we have
\begin{displaymath}
\begin{split}
& \frac{\partial^2 D_{\alpha} (A | B)}{\partial B^2} (X,X) \\
& = \frac{2}{1 - \alpha} \tr \left[ (B^{-1}X B^{-1} - \frac{1 + \alpha}{2}(\frac{ 1 -\alpha}{2} A + \frac{1 + \alpha}{2} B)^{-1} X (\frac{1 - \alpha}{2} A + \frac{1 + \alpha}{2}B)^{-1}) X \right].
\end{split}
\end{displaymath}
When $B=A$, we get
\begin{displaymath}
\left. {\frac{\partial^2 D_{\alpha} (A | B)}{\partial B^2} }%
\right|_{%
\stackunder[1pt]{$\scriptscriptstyle B=A$}{}} (X,X) = \tr(A^{-1} X A^{-1} X) = \tr((A^{-1/2} X A^{-1/2})^2) \geq 0.
\end{displaymath}
Hence, $ D_{\alpha} (A | B)$  is a quantum divergence.
\end{proof}

\begin{lemma} \label{L:divergence}
Let $A, B \in \mathbb{P}_{m}$ and $\alpha \in [-1,1]$. Then
\begin{itemize}
  \item[(i)] $D_{\alpha}(P A Q | P B Q) = D_{\alpha}(A | B)$ for any $P, Q \in \mathrm{GL}_{m}$,
  \item[(ii)] $D_{\alpha}(A^{-1} | B^{-1}) = D_{\alpha}(B | A)$, and
  \item[(iii)] $D_{\alpha}(A^{t} | B^{t}) \leq t \, D_{\alpha}(A | B)$ for any $t \in [0,1]$.
\end{itemize}
\end{lemma}

\begin{proof}
Let $A, B \in \mathbb{P}_{m}$. It is enough to show (i), (ii) and (iii) for $\alpha \in (-1,1)$ due to the limit process.
\begin{itemize}
  \item[(i)] Since $\det (X Y) = \det X \det Y$ for any square matrices $X, Y$, we obtain (i) easily.
  \item[(ii)] Since
  \begin{displaymath}
  \begin{split}
  \frac{\det \left( \frac{1 - \alpha}{2} A^{-1} + \frac{1 + \alpha}{2} B^{-1} \right)}{(\det A)^{-(1 - \alpha)/2} (\det B)^{-(1 + \alpha)/2}} & = \frac{\det \left( A^{-1} \left( \frac{1 - \alpha}{2} B + \frac{1 + \alpha}{2} A \right) B^{-1} \right)}{(\det A)^{-(1 - \alpha)/2} (\det B)^{-(1 + \alpha)/2}} \\
  & = \frac{\det \left( \frac{1 - \alpha}{2} B + \frac{1 + \alpha}{2} A \right)}{(\det A)^{(1 + \alpha)/2} (\det B)^{(1 - \alpha)/2}},
  \end{split}
  \end{displaymath}
  we get
  \begin{displaymath}
  D_{\alpha}(A^{-1} | B^{-1}) = \frac{4}{1 - \alpha^2} \log \frac{\det \left( \frac{1 - \alpha}{2} B + \frac{1 + \alpha}{2} A \right)}{(\det B)^{(1 - \alpha)/2} (\det A)^{(1 + \alpha)/2}} = D_{\alpha}(B | A).
  \end{displaymath}
  \item[(iii)] Since the map $X \in \mathbb{P}_{m} \mapsto X^{t}$ is concave for any $t \in [0,1]$,
  \begin{displaymath}
  \frac{1 - \alpha}{2} A^{t} + \frac{1 + \alpha}{2} B^{t} \leq \left( \frac{1 - \alpha}{2} A + \frac{1 + \alpha}{2} B \right)^{t}
  \end{displaymath}
  and so
  \begin{displaymath}
  \begin{split}
  \frac{\det \left( \frac{1 - \alpha}{2} A^{t} + \frac{1 + \alpha}{2} B^{t} \right)}{(\det A)^{t(1 - \alpha)/2} (\det B)^{t(1 + \alpha)/2}} & \leq \frac{\det \left( \frac{1 - \alpha}{2} A + \frac{1 + \alpha}{2} B \right)^{t}}{(\det A)^{t(1 - \alpha)/2} (\det B)^{t(1 + \alpha)/2}} \\
  & = \left[ \frac{\det \left( \frac{1 - \alpha}{2} A + \frac{1 + \alpha}{2} B \right)}{(\det A)^{(1 - \alpha)/2} (\det B)^{(1 + \alpha)/2}} \right]^{t}.
  \end{split}
  \end{displaymath}
  Thus, $D_{\alpha}(A^{t} | B^{t}) \leq t \, D_{\alpha}(A | B)$.
\end{itemize}
\end{proof}

In particular, note from \cite{Sra} that
\begin{equation} \label{E:divergence-metric}
d_{S}(A, B) = \sqrt{D_{0}(A | B)} = 2 \sqrt{\log \det \left( \frac{A + B}{2} \right) - \frac{1}{2} \log \det (A B)}
\end{equation}
is a distance on $\mathbb{P}_{m}$. Since $\log \det A = \tr \log A$ for any $A \in \mathbb{P}_{m}$, we have an alternative expression of $d_{S}(A, B)$ such as
\begin{displaymath}
d_{S}(A, B) = 2 \sqrt{ \tr \left[ \log \left( \frac{A + B}{2} \right) - \frac{\log A + \log B}{2} \right]}.
\end{displaymath}

Unlike the proof of \cite{Sra}, the following properties of the metric $d_{S}$ can be proved more easily by Lemma \ref{L:divergence}.
\begin{theorem}
The metric $d_{S}$ is invariant under congruence transformations and inversion, i.e.,
\begin{displaymath}
d_{S}(S A S^{*}, S B S^{*}) = d_{S}(A, B) = d_{S}(A^{-1}, B^{-1})
\end{displaymath}
for any $S \in \mathrm{GL}_{m}$. Moreover,
\begin{displaymath}
d_{S}(A \#_{t} B, A \#_{t} C) \leq \sqrt{t} d_{S}(B, C)
\end{displaymath}
for any $t \in [0,1]$.
\end{theorem}

It has been shown in \cite[Proposition 3.10]{CM} that for any $\mathbb{A} = (A_{1}, \ldots, A_{n}) \in \mathbb{P}_{m}^{n}$,
\begin{displaymath}
G(\mathbb{A}) = \underset{X \in \mathbb{P}_{m}}{\arg \min} \, \sum_{i=1}^{n} D_{\alpha}(A_{i} | X)
\end{displaymath}
exists uniquely, and we call it the \emph{right mean}. Indeed, the gradient of the objective function $\displaystyle \phi(X) = \sum_{i=1}^{n} D_{\alpha}(A_{i} | X)$ is
\begin{displaymath}
\nabla \phi(X) = \frac{2}{1 - \alpha} \left[ \sum_{i=1}^{n} \left( \frac{1 - \alpha}{2} A_{i} + \frac{1 + \alpha}{2} X \right)^{-1} - n X^{-1} \right]
\end{displaymath}
and its Hessian is
\begin{displaymath}
\begin{split}
& \nabla^{2} \phi(X) \\
& = \frac{2}{1 - \alpha} \left[ n X^{-1} \otimes X^{-1} - \frac{1 + \alpha}{2} \sum_{i=1}^{n} \left( \frac{1 - \alpha}{2} A_{i} + \frac{1 + \alpha}{2} X \right)^{-1} \otimes \left( \frac{1 - \alpha}{2} A_{i} + \frac{1 + \alpha}{2} X \right)^{-1} \right].
\end{split}
\end{displaymath}
Since $\displaystyle \left( (1 - \alpha) A_{i} + (1 + \alpha) X \right)^{-1} < (1 + \alpha)^{-1} X^{-1}$, we get $\nabla^{2} \phi(X) > 0$ for the local optimality of $\nabla \phi(X) = 0$. So the local minimizer becomes the global minimizer by uniqueness of $X$ satisfying $\nabla \phi(X) = 0$.
Note that $G(\mathbb{A})$ coincides with the unique solution of the following matrix equation
\begin{displaymath}
O = \nabla \phi(X) = \frac{2}{1 - \alpha} \left[ \sum_{i=1}^{n} \left( \frac{1 - \alpha}{2} A_{i} + \frac{1 + \alpha}{2} X \right)^{-1} - n X^{-1} \right].
\end{displaymath}
It is equivalent by setting $t = \frac{1 - \alpha}{2}$ that
\begin{displaymath}
X = \left[ \frac{1}{n} \sum_{i=1}^{n} \left( (1-t) X + t A_{i} \right)^{-1} \right]^{-1}.
\end{displaymath}
Hence, we can see that our mean $G_{t}(\omega; \mathbb{A})$ is a weighted version of $G(\mathbb{A})$, and many interesting properties including the monotonicity and relationship with the matrix power mean have been shown in this paper.

Finding the explicit form of $G_{t}(\omega; \mathbb{A})$ is a challengeable problem.
Remarkably it has been proved in \cite[Theorem 14]{Sra} that $G_{t}(\omega; A, B)$ for $\omega = (1/2, 1/2)$ and $t = 1/2$ is the same as the metric geometric mean of $A, B$:
\begin{equation} \label{E:Sra}
G(A, B) = \underset{X \in \mathbb{P}_{m}}{\arg \min} \, d_{S}^{2}(X, A) + d_{S}^{2}(X, B) = A \# B = G_{1/2}(1/2, 1/2; A, B).
\end{equation}
We provide the explicit form of $G_{t}(\omega; A, B)$ for a special case of $A, B$.
\begin{proposition}
Let $\omega = (w_{1}, w_{2}) \in \Delta_{2}$ and $A, B \in \mathbb{P}_{m}$ such that $(t - w_{1}) A = (w_{2} - t) B$ for either $w_{1} < t < w_{2}$ or $w_{2} < t < w_{1}$. Then
\begin{displaymath}
G_{t}(\omega; A, B) = \sqrt{\frac{t(t - w_{1})}{(1-t)(w_{2} - t)}} A.
\end{displaymath}
\end{proposition}

\begin{proof}
Let $X = G_{t}(\omega; A, B)$. We need to solve the equation
\begin{displaymath}
X^{-1} = w_{1} ((1-t) X + t A)^{-1} + w_{2} ((1-t) X + t B)^{-1}.
\end{displaymath}
It is equivalent to
\begin{displaymath}
((1-t) X + t A) X^{-1} ((1-t) X + t B) = w_{1} ((1-t) X + t B) + w_{2} ((1-t) X + t A).
\end{displaymath}
By a simple calculation and by assumption, we obtain
\begin{displaymath}
(1-t) X - t A X^{-1} B = (1-t) (A + B) - (w_{1} B + w_{2} A) = (w_{1} - t) A + (w_{2} - t) B = 0.
\end{displaymath}
Thus, $(1-t) X = t A X^{-1} B$, which is the same as $\displaystyle X A^{-1} X = \frac{t}{1-t} B$. By the Riccati lemma and the joint homogeneity of metric geometric mean,
\begin{displaymath}
X = A \# \left( \frac{t}{1-t} B \right) = \sqrt{\frac{t(t - w_{1})}{(1-t)(w_{2} - t)}} A \# A = \sqrt{\frac{t(t - w_{1})}{(1-t)(w_{2} - t)}} A.
\end{displaymath}
\end{proof}

\begin{remark}
Let $0 \leq s \leq 1/2 \leq t \leq 1$. By Theorem \ref{T:monotone} and \eqref{E:Sra}
\begin{equation} \label{E:s-1/2-t}
G_{s}(1/2, 1/2; A, B) \geq G_{1/2}(1/2, 1/2; A, B) = A \# B \geq G_{t}(1/2, 1/2; A, B).
\end{equation}
From this inequalities we give some open questions.
\begin{itemize}
\item[(1)] Let
\begin{displaymath}
A =
\left[
  \begin{array}{cc}
    2 & -1 \\
    -1 & 2 \\
  \end{array}
\right], \
B =
\left[
  \begin{array}{cc}
    3 & -2 \\
    -2 & 3 \\
  \end{array}
\right], \
C =
\left[
  \begin{array}{cc}
    2 & 1 \\
    1 & 2 \\
  \end{array}
\right]
\end{displaymath}
and $\omega = (1/3, 1/3, 1/3)$. Then
\begin{displaymath}
\begin{split}
G_{1/2}(\omega; A, B, C) & \approx
\left[
  \begin{array}{cc}
    1.96124391 & -0.53074303 \\
    -0.53074303 & 1.96124391 \\
  \end{array}
\right], \\
\Lambda(\omega; A, B, C) = (A B C)^{1/3} & \approx
\left[
  \begin{array}{cc}
    1.95423082 & -0.51198125 \\
    -0.51198125 & 1.95423082 \\
  \end{array}
\right]
\end{split}
\end{displaymath}
since $A, B, C$ commute. So $G_{1/2}(\omega; A, B, C) \neq \Lambda(\omega; A, B, C)$, in general.
On the other hand, it would be interesting to find any relationship between $G_{t}(\omega; \mathbb{A})$ and $\Lambda(\omega; \mathbb{A})$.
One can naturally ask from \eqref{E:s-1/2-t} whether the following hold:
\begin{displaymath}
\begin{split}
G_{t}(\omega; \mathbb{A}) \geq \Lambda(\omega; \mathbb{A}) \quad \textrm{if} \quad t \leq 1/2, \\
G_{t}(\omega; \mathbb{A}) \leq \Lambda(\omega; \mathbb{A}) \quad \textrm{if} \quad t \geq 1/2.
\end{split}
\end{displaymath}

\item[(2)] Since the metric geometric mean is log-majorized by the log-Euclidean mean \cite{AH}:
\begin{displaymath}
A \#_{s} B \prec_{\log} \exp \left( (1-s) \log A + s \log B \right), \quad s \in [0,1],
\end{displaymath}
one can ask from Remark \ref{R:log-Euclidean} with \eqref{E:s-1/2-t} that $G_{t}(1/2, 1/2; A^{p}, B^{p})^{1/p}$ for $t \in [1/2, 1]$ converges to the log-Euclidean mean of $A, B$ as $p \to 0^+$ in terms of the weak log-majorization. That is,
\begin{displaymath}
G_{t}(1/2, 1/2; A^{p}, B^{p})^{1/p} \nearrow_{\prec_{w \log}} \exp \left( \frac{\log A + \log B}{2} \right).
\end{displaymath}
\end{itemize}
\end{remark}

\vspace{4mm}

\textbf{Acknowledgements} \\

- \textbf{Funding}: This work was supported by the National Research Foundation of Korea (NRF) grant funded by the Korea government (MSIT) (No. NRF-2022R1A2C4001306). \\

- \textbf{Declarations}: The authors contributed equally to this paper and have no conflict of interest. \\

- \textbf{Data availability}: Not applicable.

\end{document}